\documentclass[11pt]{amsart}

\usepackage{latexsym}
\usepackage[utf8]{inputenc}
\usepackage{amsmath}
\usepackage{amssymb}
\usepackage{amsthm}
\usepackage{amsfonts}
\usepackage{mathrsfs}

\usepackage{hyperref} 
\usepackage{graphicx}

\usepackage[all]{xy}  
\usepackage{mathtools}
\usepackage{color}
\usepackage{comment}
\usepackage{appendix}
\usepackage{tikz}
\usepackage{tikz-cd}
\usepackage[english]{babel}
\usepackage{amscd,url,geometry,pdflscape}
\usepackage{comment}
\usepackage{mathalfa}

\usepackage{marginnote}

\numberwithin{equation}{section}

\newtheorem{thm}[equation]{Theorem}

\newtheorem{lem}[equation]{Lemma}

\def\kF{\mathfrak{f}}

\def\cInd{\mathrm{c\!\!-\!\!Ind}}

\def\g{\mathbf{g}}

\def\Irr{\mathrm{Irr}}
\def\sgn{\mathrm{sgn}}

\def\fo{\mathfrak{o}}
\def\fp{\mathfrak{p}}

\def\fe{\mathfrak{e}}
\def\fw{\mathfrak{w}}

\def\Cent{\mathrm{Cent}}

\def\dim{\mathrm{dim}}
\def\Gal{\mathrm{Gal}}
\def\Hom{\mathrm{Hom}}

\def\Ind{\mathrm{Ind}}

\def\val{\mathrm{v}}

\def\PGL{\mathrm{PGL}}

\def\GL{\mathrm{GL}}

\def\bW{\mathbf{W}}

\def\C{\mathbb{C}}

\def\Z{\mathbb{Z}}

\def\Qp{\mathbb{Q}_p}
\def\SL{\mathrm{SL}}

\def\PSL{\mathrm{PSL}}

\def\bbT{\mathbb{T}}

\def\W{\mathbf{W}}

\def\fo{\mathfrak{o}}
\def\fp{\mathfrak{p}}

\def\rig{{\mathrm{\rig}}}

\def\tr{{\mathrm{tr}}}

\def\trace{\mathrm{trace}}

\def\rig{{\mathrm{rig}}}
\def\rN{{\mathrm{N}}}

\newcommand{\sS}{\mathscr{S}}

\def\g{\mathbf{g}}

\DeclareSymbolFont{bbold}{U}{bbold}{m}{n}
\DeclareSymbolFontAlphabet{\mathbbold}{bbold}

\begin{document}

\title[Endoscopy]{An introduction to endoscopy via $\SL_2$}

\author[A.-M. Aubert]{Anne-Marie Aubert}
\address{Sorbonne Universit\'e and Universit\'e Paris Cit\'e, CNRS, IMJ-PRG, 
F-75005 Paris, France}
\email{anne-marie.aubert@imj-prg.fr}
\author[R. Plymen]{Roger Plymen}
\address
{School of Mathematics, Southampton University, Southampton SO17 1BJ,  England 
\emph{and} School of Mathematics, Alan Turing Building, Manchester University, Manchester M13 9PL, England}
\email{r.j.plymen@soton.ac.uk, roger.j.plymen@manchester.ac.uk}

\maketitle

\begin{abstract}  This article provides an explicit introduction to the theory of endoscopy for $p$-adic groups.   We will focus on the single $p$-adic group $\SL_2$.   
We consider the depth-zero supercuspidal $L$-packets of  $\SL_2(F)$ where $F$ is a non-archimedean local field of characteristic zero.   
We compare the explicit endoscopic character identities for $\SL_2(F)$ with  the classical character formulas of Sally-Shalika.   The formulas which appear in Theorems \ref{A} and \ref{B} serve to illustrate, in a very explicit 
way, the theory of endoscopic character identities.

\end{abstract}

\tableofcontents

\section{Introduction} To a finite-dimensional complex representation $(\pi,V)$  of a group $G$ one
can associate the character function
\[
\Theta \colon  G \to \C, \quad \quad \Theta(g) = \tr(\pi(g)).
\]
This function carries a lot of information about the representation. A basic
fact in the representation theory of \emph{finite} (and more generally \emph{compact topological})
 groups is that every irreducible representation is finite-dimensional and its
isomorphism class is determined by the character function.

In this article we will restrict ourselves to the $p$-adic group $\SL_2$.  This group is, of course,  a non-compact topological group.     All the representations which we consider, even though infinite-dimensional,
have a robust theory of characters developed by Harish-Chandra \cite{Ka5}.

 We will encounter the characters of certain finite-dimensional representations, the Harish-Chandra characters of certain infinite-dimensional representations of $\SL_2$, and the stable characters 
attached to the  \emph{endoscopic groups}.    The theory of the endoscopic character identities blends all these characters in a subtle and sophisticated way.

We will focus on the depth-zero supercuspidal $L$-packets of  $\SL_2$.    We compare the explicit endoscopic character identities for $\SL_2$ with  the classical character formulas of Sally-Shalika.
  
  The formulas which appear in Theorems \ref{A} and \ref{B} serve to illustrate, in a very explicit way, the theory of endoscopic character identities.  We hope that our exposition will enable readers to obtain a foothold on this fascinating chapter in representation theory.
  The subject is quite technical, but we have tried to make our exposition accessible to the general reader.

\subsubsection{Notation and background}  Let $F$ be a non-archimedean local field of characteristic zero, so that $F$ is a finite extension of $\Qp$.    We shall assume that $p > 2$. 
Let $q = q_F$ denote the cardinality of the residue field of $F$. Let $\val(x)$ denote the normalized valuation of $x \in F^\times$.  Let $\fo_F$ denote the ring of integers of $F$,  let $\varpi$ be a uniformiser for $F$,  let $\fp_F=\varpi\fo_F$ be the maximal ideal of $\fo_F$, and let  $\varepsilon$ be a fixed nonsquare element in $\fo_F^\times$.  In that case, $\{1,\varpi,\varepsilon,\varepsilon\varpi\}$  is a set of representatives of $F^\times/(F^\times)^2$. 

Let $G = \SL_2(F)$.  We shall focus on the elliptic torus
\begin{equation} \label{eqn:tori}
T^\varepsilon:=
\left\{\left(
\begin{matrix}
a&b\cr
b\varepsilon&a
\end{matrix}\right)\;:\; a,b\in\fo_F, a^2-\varepsilon b^2=1
\right\}
\subset \SL_2(F).
\end{equation}
This elliptic torus $T^\varepsilon$ is a representative of the single stable conjugacy class of unramified elliptic maximal $F$-tori in $\SL_2(F)$.   

We denote by $T^\varepsilon_{\mathrm{reg}}$ the set of elements in $T^\varepsilon$ with $b \neq 0$.    We will now define $f$ as follows:
\begin{equation} \label{eqn:definition-of-f}
f \colon T^\varepsilon_{\mathrm{reg}} \to \Z, \quad \quad 
\left(
\begin{matrix}
a&b\cr
b\varepsilon&a
\end{matrix}\right) \mapsto (-q)^{\val(b)} .
\end{equation}
The image of $f$ is $\{(-q)^n : n \geq 0\}$.

We denote by $T^{\varepsilon}_1$ the pro-unipotent radical of $T^\varepsilon$, namely
\begin{equation} \label{T_1}
T^\varepsilon_1:=
\left\{\left(
\begin{matrix}
a&b\cr
b\varepsilon &a
\end{matrix}\right) \in T^\varepsilon \;:\; a \in 1 + \fp_F,\; b \in \fp_F \; \right\}.
\end{equation}

\subsubsection{The local Langlands conjecture}  Our reference for the local Langlands conjecture is the article by Taibi \cite{T}, especially sections 6.1 and 6.2.   Since we are dealing 
exclusively with the depth-zero representations of $\SL_2(F)$, there are, naturally, very many simplifications.  For $\SL_2$, the local Langlands conjecture is a theorem, namely the local Langlands 
correspondence for $\SL_2(F)$, see \cite{LL} or \cite{HS}.      
The coincidence for $\GL_n(F)$ of Kaletha’s construction in \cite{Ka2} of the local Langlands correspondence for regular supercuspidal representations, which generalizes that of \cite{DR},  with Harris--Taylor's construction has been established in \cite{OT}. 
Regular supercuspidal representations are supercuspidal representations which are obtained by Yu’s construction \cite{Yu} and satisfy a certain regularity condition. Every depth-zero supercuspidal representation of $\GL_n$ is regular (and every supercuspidal representation of $\GL_n$ is regular when $p$ does not divide $n$).

Let $\W_F$ denote the Weil group of $F$.   If $G = \SL_2(F)$ then the Langlands dual group $\widehat{G}$ is $\PGL_2(\C)$ and the $L$-group ${}^LG$ is $\PGL_2(\C) \rtimes \W_F$.   The cross product here is actually the direct product.   
In general, a Langlands parameter is a homomorphism from $\W_F \times \SL_2(\C)$ to ${}^LG$ satisfying several conditions.  The set of equivalence classes of Langlands parameters is denoted $\Phi(G)$.   

The set of equivalence classes of smooth irreducible representations of $G$ is denoted $\Pi(G)$.   Then the local Langlands correspondence for $G$ is a map
\[
\mathrm{LL}: \Pi(G) \to \Phi(G)
\]
satisfying several conditions, itemized in \cite[\S6.1, \S6.2]{T}.  We note that Taibi makes a distinction between the crude local Langlands correspondence in \S6.1 and the refined local Langlands correspondence in \S6.2.   One of the conditions for the refined local Langlands correspondence is the endoscopic character identity, see equation (6.4) in \cite[\S6.2]{T}.  The endoscopic character identity for $\SL_2$ is the subject of this paper.   

Given $\phi \in \Phi(G)$ we denote the pre-image by $\Pi_\phi$:
\begin{eqnarray*}
\Pi(G)  \to &\Phi(G)\\
\Pi_\phi  \to &\phi.   
\end{eqnarray*}
Then $\Pi_\phi$ is called an $L$-packet.   

 In this paper, we only consider Langlands parameters with trivial $\SL_2(\C)$-part, so that we have 
\[
\phi \colon \W_F \to {}^LG.
\]  
Let $S_\phi$  denote the centralizer in $\PGL_2(\C)$ of the image of $\phi$, and $\Irr(S_\phi)$ the set of isomorphism classes of irreducible representations of $S_\phi$.   
We have the unipotent subgroup
\[
U(F) =\left\{ \left(
\begin{matrix}
1 & x \cr
0 & 1
\end{matrix}\right) : x \in F \right \} \subset \SL_2(F).
\]
   A Whittaker datum $\fw$  for $\SL_2(F)$  is a pair $(U, \lambda)$ where   $\lambda \colon U(F) \to \C^\times$ is a non-trivial character.     
 An irreducible smooth representation $(\pi, V)$ of $\SL_2(F)$  is called $\fw$-generic if 
 \[
 \Hom_{U(F)}(V, \C_\lambda) \neq 0.
 \]
 In Section~\ref{sec:depth-zero}, we construct an explicit Whittaker datum $\fw$ which is well-suited to the depth-zero supercuspidal representations 
 of $\SL_2(F)$.    This datum $\fw$ remains fixed throughout this article.    

We now refer to Conjecture 6.4 and Conjecture 6.5 (Shahidi)  in the section on the refined local Langlands conjecture in \cite[\S6.2]{T}.   Let
$\phi$ be a discrete parameter.   For $\SL_2(F)$, we have, in accordance with these conjectures, a  bijective map
\begin{equation}
\iota_\phi  \colon \Pi_\phi(\SL_2(F))  \to \Irr(S_\phi)
\end{equation}
that sends the $\fw$-generic representation of $\SL_2(F)$ to the trivial character of $S_\phi$.

\subsubsection{Virtual characters}   Given  $s \in S_\phi$, we attach to $(\phi,s)$ the  virtual character    
\begin{equation} \label{eqn:virtual}                                                       
\Theta_{\phi, s} := \sum_{\pi\in\Pi_\phi(\SL_2(F)) } \trace(\iota_\phi(\pi))(s) \cdot \Theta_{\pi}.
\end{equation}

Note that this expression is a blend of characters: the map $s \mapsto \trace(\iota_\phi(\pi))(s)$ is the character of the finite-dimensional representation $\iota_\phi(\pi)$ of $S_\phi$ and 
$\Theta_\pi$ is the Harish-Chandra character of the infinite-dimensional representation $\pi$ of $\SL_2(F)$.       

Let $\widehat{H}$ denote the identity component of the centralizer of $s$ in $\widehat{G}$. The \emph{extended endoscopic triple}, as defined in \cite[\S6.2]{T},  is 
        \[
        \fe(s): = (s, H, {}^L\eta).
        \]     Then $\phi$ will factorize as follows:
       \[
       \begin{CD}
       \W_F @> \phi^H >> {}^LH @> {}^L\eta >> {}^LG.
       \end{CD}
       \]
           The map $\phi^H$ will be called the  \emph{little parameter} associated to $\phi$.   This  leads to the definition of the \emph{stable character}.   
 The stable character $\sS\Theta_{\phi^H}$ is defined as 
 \[
 \sS\Theta_{\phi^H} := \Theta_{\phi^H, 1}.
 \] 
 
\subsubsection{The endoscopic character identities}   Let $\phi$ be a tempered Langlands parameter.   We now make the following definition:
\[
\mathscr{E}(\phi, s:\gamma): =  \sum_{\delta \in H}  \Delta[\fw, \fe(s)](\delta, \gamma) \sS\Theta_{\phi^{H}} (\delta) 
\]
where $\Delta[\fw,\fe(s)]$ is the Langlands-Shelstad transfer factor, $\gamma$ is any strongly regular semisimple element of $G$ and 
$\delta \in H$ in the sum runs over all stable conjugacy classes of regular semisimple elements of $H(F)$.    Then the  endoscopic character identity for $\SL_2$ is     
\begin{eqnarray}\label{endo}
  \mathscr{E}(\phi, s:\gamma) = \Theta_{\phi,s}(\gamma). 
\end{eqnarray}       

 In the context of quasi-split connected reductive groups, our reference for the endoscopic character identity is equation (6.4) in the article by Taibi \cite[\S6.2]{T}.       
  From a purely notational point of view, we prefer to switch $\gamma$ and $\delta$ in \cite[Equation (6.4)]{T}.
 
 The explicit endoscopic character identities, for the depth-zero supercuspidal $L$-packets of $\SL_2$, are stated in Theorems \ref{A} and \ref{B}.


The characters of the discrete series representations of $\SL_2(F)$  were obtained by Sally-Shalika \cite{SS}.     In particular, for the depth-zero supercuspidal $L$-packets, the characters 
appear in Table 3 of \cite{SS}.  They are given as locally integrable functions on the split maximal torus of $\SL_2(F)$ and on the elliptic tori in $\SL_2(F)$.   Many years later, proofs were supplied by  Adler, DeBacker, Sally and Spice in \cite{ADSS}.   For the most part, we find it more convenient to refer to \cite{ADSS}. 


The elliptic torus $T^\varepsilon$ admits a unique nontrivial character $\psi$ such that $\psi = \psi^{-1}$.  This  character will be denoted $\psi_0$.   
The \emph{regular supercuspidal parameters} are defined by Kaletha \cite[Def. 5.2.3]{Ka2}.   The group $\SL_2(F)$ admits a unique nontrivial non-regular parameter for the depth-zero supercuspidal representations, see Theorem \ref{B}.    

  When we combine Theorems \ref{vch}  and  \ref{ENDO}, we obtain the following result.


\begin{thm}\label{A}    Let $\gamma$ be a regular, semisimple element in $T^\varepsilon$.  Let $\Pi_\phi$ be a depth-zero supercuspidal $L$-packet of cardinality $2$.       
This $L$-packet is constructed from a depth-zero character $\psi$ of $T^\varepsilon$ such that $\psi \neq \psi^{-1}$.         
 Let $s$ denote the non-trivial element in $S_\phi$.  
 We have  
  \begin{eqnarray*}\label{AA}
 \mathscr{E}(\phi, s :\gamma) = \Theta_{\phi, s}(\gamma) =  - f(\gamma)(\psi(\gamma) + \psi(1/\gamma)).
\end{eqnarray*}
\end{thm}

Let $s_1, s_2$ and $s_3$ denote the images in $\PGL_2(\C)$ of the matrices
\[
\left(\begin{matrix}1&0\\0&-1\end{matrix}\right), \quad \left(\begin{matrix}0&1\\-1&0\end{matrix}\right),  \quad   \left(\begin{matrix}0&1\\1&0\end{matrix}\right).
\] 
When we combine Theorems \ref{FAR-NON-REG}, 
\ref{NEAR-NON-REG}  and  \ref{abc}, we obtain the following result.  

\begin{thm} \label{B}   Let $\gamma$ be a regular, semisimple element in $T^\varepsilon$.  Let $\Pi_\phi$ be the unique depth-zero supercuspidal $L$-packet of cardinality $4$.     
Let $\psi_0$ denote the unique nontrivial quadratic character of $T^\varepsilon$.   We have 
\begin{eqnarray*}
\mathscr{E}(\phi, s_1 : \gamma) &=&  \Theta_{\phi, s_1}(\gamma) = -2 f(\gamma) \psi_0(\gamma) \\
\mathscr{E}(\phi, s_2 : \gamma) &=& \Theta_{\phi, s_2}(\gamma) = 0\\
\mathscr{E}(\phi, s_3:\gamma) & = & \Theta_{\phi, s_3}(\gamma) = 0.
\end{eqnarray*}  
\end{thm} 
In order to prove  Theorems \ref{A} and \ref{B}, we computed the right-hand-side of the endoscopic character identity via the Sally-Shalika formulas, and the left-hand-side via the basics of endoscopy.    
   
 In Theorem \ref{B}, there are, in the background, three endoscopic groups $H_1, H_2, H_3$ corresponding to the three quadratic extensions of $F$.

  \textbf{Acknowledgement}. We would like to thank Loren Spice for a valuable exchange of emails, Shaun Stevens for valuable discussions, Tasho Kaletha for  sending us constructive criticism, and the referee for providing many insightful comments.

\section{Elliptic tori}   
\subsubsection{Elliptic tori} There are two tori relevant for us:
 Let $\varepsilon$ be a fixed nonsquare element in  $\fo_F^{\times}$, then

\begin{equation} \label{eqn:Torus-Tepsilon}
T^\varepsilon:=
\left\{\left(
\begin{matrix}
a&b\cr
b\varepsilon &a
\end{matrix}\right)\;:\; a,b\in\fo_F, a^2-\varepsilon b^2=1
\right\}
\end{equation}
and 
\begin{equation} \label{eqn:Torus-Tepsilon-varpi}
T^{\varepsilon, \varpi} :=
\left\{\left(
\begin{matrix}
a&b\varpi^{-1} \cr
b\varepsilon \varpi &a
\end{matrix}\right)\;:\; a,b\in\fo_F, a^2-\varepsilon b^2=1
\right\}.
\end{equation}
These two tori admit the same splitting field, namely $F(\sqrt \varepsilon)$ and are therefore stably conjugate, or, equivalently,
$\GL_2(F)$-conjugate,    thanks to the following identity in $\GL_2(F)$:
\begin{equation}\label{ggg}
\left(\begin{array}{cc}
1 & 0\\
0 & \varpi
\end{array}
\right)
\left(\begin{array}{cc}
a & b\\
\varepsilon b & a
\end{array}
\right)
\left(\begin{array}{cc}
1 & 0\\
0 & \varpi^{-1}
\end{array}
\right)
 = 
\left(\begin{array}{cc}
a & \varpi^{-1}b\\
\varepsilon \varpi b & a
\end{array}\right).
\end{equation}
There is a  single stable conjugacy class of unramified elliptic maximal $F$-tori,  represented by $T^\varepsilon$.   It splits into two 
distinct $G$-conjugacy classes, represented by $T^\varepsilon$ and $T^{\varepsilon, \varpi}$.   

Let $E^1$ denote the group of norm $1$ elements in $E$.   We have, following \cite[\S3.1]{ADSS}, canonical isomorphisms of $F$-tori:
\[
T^\varepsilon \to E^1, \quad \left(
\begin{matrix}
a&b\cr
b\varepsilon &a
\end{matrix}\right)
\mapsto a + b\sqrt\varepsilon
\]
\[
T^{\varepsilon, \varpi} \to E^1, \quad \left(
\begin{matrix}
a&b\varpi^{-1}\cr
b\varepsilon \varpi &a
\end{matrix}\right)
\mapsto a + b\sqrt\varepsilon.
\]
Given $a + b\sqrt \varepsilon \in E^1$ we shall, following \cite[\S3.1]{ADSS} write 
\[
\mathrm{Im}_\varepsilon(a + b\sqrt \varepsilon) = b.
\]  
We have
\begin{eqnarray}\label{gdot}
\mathrm{Im}_\varepsilon (\g \cdot \gamma) =  \mathrm{Im}_\varepsilon (\gamma),
\end{eqnarray}
for every $\gamma\in T^{\varepsilon}$, where 
\begin{equation} \label{eqn:g}
\mathbf{g}: = \left(
\begin{array}{cc}
1 & 0\\
0 & \varpi
\end{array}
\right) \in \GL_2(F)
\end{equation}
and $\mathbf{g} \cdot \gamma: = \mathbf{g} \gamma \mathbf{g}^{-1}$.
The equation \eqref{ggg} can be rephrased as the following statement:
\begin{equation} \label{eqn:tori}
\g \cdot T^\varepsilon =  \g T^\varepsilon \g^{-1} = T^{\varepsilon, \varpi}.
\end{equation}

The sign character $\sgn_\varepsilon \colon F^\times \to \{\pm 1\}$ is determined by the short exact sequence
  \[
  1 \to \rN_{E/F}E^\times \to F^\times \to \Gal(E/F) \to 1.
  \]

We note that
 \[
 \sgn_\varepsilon(x) = (-1)^{\val(x)}
 \]
 for all $x \in F^\times$, see (2.1) in \cite {ADSS}. Then we have, using (\ref{gdot}),
 \begin{eqnarray}\label{SGN}
 \sgn_\varepsilon(\varpi^{-1} \mathrm{Im}_\varepsilon(\g \cdot \gamma))  &=& \sgn_\varepsilon(\varpi^{-1} \mathrm{Im}_\varepsilon(\gamma)) \\
  &=& \sgn_\varepsilon(\varpi^{-1}) \sgn_\varepsilon( \mathrm{Im}_\varepsilon(\gamma)) \nonumber \\
 &=& (-1)^{\val(\varpi^{-1})}\sgn_\varepsilon  (\mathrm{Im}_\varepsilon(\gamma)) \nonumber\\
 &=&- \sgn_\varepsilon( \mathrm{Im}_\varepsilon(\gamma)). \nonumber
 \end{eqnarray}

\subsubsection{Weyl discriminant} In \cite[Definition 3.7]{ADSS}, the Weyl discriminant $D_G$ is defined as follows:  
\[
D_G \left(
\begin{array}{cc}
a & b\\
c & d
\end{array}
\right)
= (a + d)^2 - 4.
\]
Now
\[
\gamma = \left(
\begin{matrix}
a&b\cr
b\varepsilon &a
\end{matrix}\right) \in T^\varepsilon
\implies D_G(\gamma) =  4a^2 - 4 = 4 \varepsilon b^2.
\]
We recall that $p \neq 2$ so that
\[
|D_G(\gamma)|^{1/2}  = |b|.
\]
An element $\gamma \in G$ is \emph{regular semisimple} if and only if $D_G(\gamma) \neq 0$.   So, an element $\gamma \in T^\varepsilon$ is regular semisimple if and only if  $b \neq 0$.

In this article, we will find the following notation useful:
\[
d_G(\gamma): = |D_G(\gamma)|^{1/2}.
\]

\subsubsection{The function $f$} If $\gamma$ is a regular semisimple element of $T^\varepsilon$ then we define
\begin{equation} \label{eqn:fgamma}
f(\gamma):= \frac{\sgn_\varepsilon(\mathrm{Im}_\varepsilon(\gamma))}{d_G(\gamma)} .
\end{equation}

In that case, we have
\begin{eqnarray*} 
f(\gamma)&=& \frac{\sgn_\varepsilon(b)}{|b|}\\
&=& (-1)^{\val(b)} q^{\val(b)}\\
&=& (-q)^{\val(b)}.
\end{eqnarray*}

It is worth noting that $f$ is invariant under the Galois group $\Gal(F(\sqrt \varepsilon)/F)$.  If $c = a + b\sqrt \varepsilon$ then we have
\begin{equation}\label{fGalois}
f(\overline{c}) = f(a - b\sqrt \varepsilon) = (-q)^{\val(-b)} = (-q)^{\val(b)} = f(c).
\end{equation}

We denote by $T^{\varepsilon}_1$ the pro-unipotent radical of $T^\varepsilon$, as in (\ref{T_1}).   
We will say, following \cite[\S15.1, \S15.2]{ADSS}, that  $\gamma$ is \emph{near the identity} if $\gamma \in T_1^\varepsilon$,
and \emph{far from the identity} if
\[
\gamma \notin Z(G)T_1^\varepsilon = T_1^\varepsilon \cup - T_1^\varepsilon,
\] 
though the terminology \emph{far from the centre} would seem to be preferable to \emph{far from the identity}.  

The following Lemma is proved in \cite[p.65]{ADSS}.    We offer a short proof.

\begin{lem}\label{farfrom}  Let $\gamma$ be a regular semisimple element 
in the elliptic torus $T^\varepsilon$ which is far from the identity.   Then we have $f(\gamma) = 1$.
\end{lem}
\begin{proof} Writing \(\gamma=a+b\sqrt{\varepsilon}\), the relation \(a^2-\varepsilon b^2=1\) shows that \(b\in\mathfrak p_F\) would imply \(\gamma\equiv\pm1\pmod{\mathfrak p_E}\), contrary to \(\gamma\notin Z(G)T_1^\varepsilon\); hence \(b\in\mathfrak o_F^\times\) and \(f(\gamma)=(-q)^{v(b)}=1\). 
\end{proof}

\subsubsection{Elliptic torus in $\SL_2(\kF)$}   The quotient $T^\varepsilon/T_1^\varepsilon$ is the unique-up-to-conjugacy maximal elliptic $\kF$-torus $\bbT^\epsilon(\kF)$ in $\SL_2(\kF)$:
\begin{equation}
\bbT^\varepsilon (\kF)=\left\{ \left(\begin{matrix}
   a & b\\
 \epsilon b & a
 \end{matrix}
 \right)\,:\, \text{$a,b\in \kF$ and $a^2-\epsilon b^2=1$}\right\}
 \end{equation} 
 where $\kF$ is the residue field of $F$.   

  
The group $E^1$ of norm one elements in $E = F(\sqrt \varepsilon)$ admits a unique nontrivial quadratic character $\psi_0$ of depth zero which therefore factors through $\bbT^\epsilon$.  
An explicit formula for $\psi_0$ (which we shall not need) is given in \cite[\S 9.2]{ADSS}.

\section{Depth-zero supercuspidal representations of $\SL_2(F)$}  \label{sec:depth-zero}
 We define  $G_{x_1}$ to be $\SL_2(\fo_F)$ and $G_{x_2}$ to be its conjugate under the element $\mathbf{g}$ of $\GL_2(F)$ defined in \eqref{eqn:g}. Then $\{G_{x_1}, G_{x_2}\}$ represent the $G$-conjugacy classes of maximal parahoric subgroups.   
The reductive quotients of $G_{x_1}$ and $G_{x_2}$ are both  isomorphic to  $\SL_2(\kF)$,  and $\{G_{x_1},G_{x_2}\}$ is a set of representatives of the maximal parahoric subgroups of $\SL_2(F)$. Let $G_{x_i}^+$ denote the pro-$p$ unipotent radical of $G_{x_i}$, for $i=1,2$.


Regular case: Let $\psi$ be a character of $\bbT^\epsilon(\kF)$ such that $\psi \ne \psi^{-1}$.  Let $R_{\bbT^\epsilon,\psi}^{\SL_2}$ be the corresponding Deligne-Lusztig virtual character. We set $\overline\kappa=-R_{\bbT^\epsilon,\psi}^{\SL_2}$. For  $i=1,2$, we
denote by  $\kappa_i$ the inflation of $\overline\kappa$ to $G_{x_i}$, and we define 
\begin{equation} \label{eqn:reg}
\pi^i:= c\!-\!\Ind_{G_{x_i}}^G \kappa_i.
\end{equation}
The representations $\pi^1$ and $\pi^2$  are depth-zero regular supercuspidal  irreducible representations of $G$ in the terminology of \cite{Ka2}.

Non-singular case: Let $\psi_0$ be the unique nontrivial quadratic character of $\bbT^\epsilon(\kF)$. The Deligne-Lusztig virtual character $R_{\bbT^\epsilon,\psi_0}^{\SL_2}$  decomposes as 
\begin{equation}
R_{\bbT^\epsilon,\psi_0}^{\SL_2}=-R_{\bbT^\epsilon,\psi_0}^+ - R_{\bbT^\epsilon,\psi_0}^-,
\end{equation}
where $R_{\bbT^\epsilon,\psi_0}^\pm$ are irreducible cuspidal representations of $\SL_2(\kF)$ of dimension $\frac{q-1}{2}$ (see \cite[Table~12.1]{DM}). We set 
\[\overline \kappa_+:=R_{\bbT^\epsilon,\psi_0}^+ \quad\text{and}\quad \overline \kappa_-:=R_{\bbT^\epsilon,\psi_0}^- .
\] 
We denote by $\kappa_\pm$  the inflation of $\overline\kappa_\pm$ to $G_{x_1}$, and we write
\begin{equation} \label{eqn:nonsingular}
\omega^+:=\cInd_{G_{x_1}}^G \kappa_+ \quad\text{and}\quad
\omega^-:=\cInd_{G_{x_1}}^G \kappa_-.
\end{equation}
We define
 \begin{equation*}
 (\g \cdot \omega^+)(x):= \omega^+(\g^{-1}x\g)\quad\text{and}\quad
 \g \cdot x:= \g x\g^{-1},
 \end{equation*}
 so that 
 \[
 (\g \cdot \omega^+)(x) = \omega^+(\g^{-1} \cdot x).
 \]
 We set 
 \begin{equation} \label{eqn:packet}
\pi_1 = \omega^+, \quad \pi_2 = \omega^-,\quad \pi_3 = \g \cdot \omega^+, \quad \pi_4 = \g\cdot \omega^- .
\end{equation}
The representations $\pi_1$, $\pi_2$, $\pi_3$, $\pi_4$ are depth-zero supercuspidal  irreducible representations of $\SL_2(F)$, they are called ``exceptional" in \cite{ADSS}, and are non-singular non-regular  in the terminology of  \cite{Ka3} (see also \cite{Au2}).


\subsubsection{The Whittaker datum}  Let
\[
U(F) = \left\{\left(
\begin{matrix}
1 & x \cr
0 & 1
\end{matrix}\right) : x \in F \right\}
\]
be the standard unipotent subgroup of $\SL_2(F)$.    Let
\[
K=G_{x_1}=\mathrm{SL}_2(\fo_F),\qquad
U_0=U(F)\cap K=U(\fo_F).
\]
Let us identify \(\overline U=U(\kF)\) with the additive group of \(\kF\). 
The sum $\overline\kappa_+\oplus\overline\kappa_-$ has degree \(q-1\) and character value \(-1\) at every nonidentity element of \(\overline U\). Consequently,
\[
(\overline\kappa_+\oplus\overline\kappa_-)|_{\overline U}
\cong
\bigoplus_{\substack{\alpha\in\widehat{\overline U}\\\alpha\ne1}}\alpha.
\]
Thus every nontrivial character of $\overline U$ occurs in exactly one constituent, with multiplicity one.

Choose a nontrivial character \(\overline\lambda\) occurring in the restriction of $\overline\kappa_+$ to $\overline U$. Now choose \(\lambda\) on \(U(F)\), trivial on \(U(\mathfrak p_F)\), inducing \(\overline\lambda\). Set \(\mathfrak w=(U,\lambda)\).

\begin{lem}\label{pi_1gen}  The representation $\pi_1$ is $\fw$-generic.
\end{lem}
\begin{proof} We have, by construction of $\overline \lambda$, 
\[
\operatorname{Hom}_{\overline U}
   (\overline\kappa_+,\overline\lambda) \neq 0.
\]
Inflation gives
\[
\operatorname{Hom}_{U(\fo_F)}
(\kappa_+,\lambda)\ne0.
\]
We now apply the induction--restriction argument.
The identity double coset in the Mackey decomposition contributes
\[
\mathrm{c\!-\!Ind}_{U(\fo_F)}^{U(F)}\kappa_+
\]
as a direct summand of \(\pi_1|_{U(F)}\). Frobenius reciprocity gives
\[
\operatorname{Hom}_{U(F)}
\left(
\mathrm{c\!-\!Ind}_{U(\fo_F)}^{U(F)}\kappa_+,\lambda
\right)
\cong
\operatorname{Hom}_{U(\fo_F)}(\kappa_+,\lambda)
\ne0.
\]
Therefore
\[
\operatorname{Hom}_{U(F)}(\pi_1,\lambda)\ne0.
\]
So the proof establishes directly that \(\pi_1\) is \((U,\lambda)\)-generic.
\end{proof}

\begin{lem}\label{pi^1gen}  The representation $\pi^1$ is $\fw$-generic.
\end{lem}
\begin{proof}  We have 
\[
\pi^1=\mathrm{c\!-\!Ind}_K^G\kappa_1.
\]
 Now choose $\overline{\lambda}$ as above.   We shall prove that 
\[
\operatorname{Hom}_{\overline U}
   (\overline\kappa,\overline\lambda) \neq 0.
\]
The regular cuspidal character satisfies
\[
\operatorname{tr}\overline\kappa(1)=q-1,\qquad
\operatorname{tr}\overline\kappa(u)=-1
\quad(1\ne u\in\overline U).
\]
Consequently,
\[
\begin{aligned}
\dim  \operatorname{Hom}_{\overline U}
   (\overline\kappa,\overline\lambda)&=\frac1q\sum_{u\in\overline U}
  \operatorname{tr}\overline\kappa(u)\,
  \overline\lambda(u)^{-1}\\
&=\frac1q\left(q-1-
  \sum_{u\ne1}\overline\lambda(u)^{-1}\right)\\
&=\frac1q(q-1+1)=1.
\end{aligned}
\]
Here we used that the sum of a nontrivial character over \(\overline U\) is zero. Inflating gives
\[
\operatorname{Hom}_{U_0}(\kappa_1,\lambda|_{U_0})\ne0.
\]
Now apply induction--restriction.
In the Mackey decomposition of \(\pi^1|_{U(F)}\), the identity double coset in \(K\backslash G/U(F)\) contributes the direct summand
\[
\mathrm{c\!-\!Ind}_{U_0}^{U(F)}(\kappa|_{U_0}).
\]
Since \(U_0\) is open in \(U(F)\), Frobenius reciprocity gives
\[
\operatorname{Hom}_{U(F)}
\left(
\mathrm{c\!-\!Ind}_{U_0}^{U(F)}\kappa_1,\lambda
\right)
\cong
\operatorname{Hom}_{U_0}(\kappa_1,\lambda|_{U_0})
\ne0.
\]
Composing such a functional with the projection onto this Mackey summand proves
\[
\operatorname{Hom}_{U(F)}(\pi^1,\lambda)\ne0.
\]
Thus \(\pi^1\) is \(\mathfrak w\)-generic.
\end{proof}

One Whittaker datum works in both cases:  $\pi^1\ \text{and}\ \pi_1 \text{ are generic for the same }\mathfrak w$.

\subsubsection{} For a fixed Whittaker datum, an \(L\)-packet of \(\mathrm{SL}_2(F)\) has at most one generic member. This is a theorem for quasi-split classical groups; see Atobe \cite[\S3.3]{At}.   Our direct induction--restriction argument has already established that \(\pi^1\) and \(\pi_1\) are \(\mathfrak w\)-generic, in their respective packets. We therefore have
\[
\operatorname{Hom}_{U(F)}(\pi^2,\lambda)=0
\]
in each regular packet, and
\[
\operatorname{Hom}_{U(F)}(\pi_j,\lambda)=0
\qquad(j=2,3,4)
\]
in the exceptional packet.

\subsubsection{The additive character $\Lambda$} We note that the character $\lambda$ determines an additive  character $\Lambda$ via the equation
\[
\Lambda(x) = \lambda \begin{bmatrix} 1 & x\\
0 & 1 \end{bmatrix}
\]
for all $x \in F$.

\section{Virtual characters far from the identity: the non-regular parameter} \label{sec:4}  We recall that  $G = \SL_2(F), \, \widehat{G} = \PGL_2(\C) = \PSL_2(\C)$.   
Let $E/F$ be the (unramified) quadratic extension of $F$ containing $\sqrt \varepsilon$ and let $K/F$ be the biquadratic extension of $F$ containing $\sqrt \varepsilon$ and $\sqrt \varpi$.    
For the Galois group  we have $\Gal(K/F) = \{1,\sigma, \tau, \sigma \tau\}$ with $\sigma \sqrt \varepsilon = - \sqrt \varepsilon$ and $\tau \sqrt \varpi = - \sqrt \varpi$.

We recall that  $s_1, s_2$  denote the images in $\PGL_2(\C)$ of the matrices
\[
\left(\begin{matrix}1&0\\0&-1\end{matrix}\right), \quad \left(\begin{matrix}0&1\\-1&0\end{matrix}\right).
\] 
  The Langlands parameter 
 \begin{eqnarray}\label{biquad}
 \phi \colon \bW_F \to \PGL_2(\C) 
 \end{eqnarray}
  is now defined   by the condition that $\phi$ factors through $\Gal(K/F)$ and satisfies the equations 
 \begin{eqnarray}\label{pd1}
  \phi(\sigma) &=& s_2, \quad   \phi(\tau) = s_1.
  \end{eqnarray}  
  We consider the associated supercuspidal $L$-packet  
\[
\Pi_\phi(\SL_2(F)) = \{\pi_1, \pi_2, \pi_3, \pi_4\},
\]
and let $\Theta_j := \Theta_{\pi_j}$ with $1 \leq j \leq 4$. 


\begin{thm}\label{FAR}  Let $\gamma$ be a regular semisimple element 
in the elliptic torus $T^\varepsilon$ which is far from the identity. Then we have the character formulas
 \begin{eqnarray*}
\Theta_1(\gamma)  &=&  - \psi_0(\gamma)\\
 \Theta_2(\gamma)  &=& -\psi_0(\gamma)\\
  \Theta_3(\gamma)  &=&  0\\
 \Theta_4(\gamma)  &=& 0.
 \end{eqnarray*}    
 \end{thm} 
 \begin{proof}   We refer to equations $(*)$ and $(**)$ on p.65 of \cite{ADSS}.  It follows immediately that 
  \[
 \Theta_{\pi_1}(\gamma) =  \Theta_{\pi_2}(\gamma)  = - \psi_0(\gamma).
 \]
 We have, for all $\gamma \in T^\varepsilon$, 
 \begin{eqnarray*}
 \Theta_{\pi_3}(\gamma) &=& \Theta_{\g \cdot \pi_1}(\gamma)\\
 &=& \Theta_{\pi_1}(\g^{-1}\cdot \gamma)\\
 &=& 0
 \end{eqnarray*}
 since $\Theta_{\pi_1}$ vanishes on $\g^{-1} \cdot T^\varepsilon$ by \cite[Theorem 15.1]{ADSS}.   
 
Similarly, we have, for all $\gamma \in T^\varepsilon$, 
\begin{eqnarray*}
 \Theta_{\pi_4}(\gamma) &=& \Theta_{\g \cdot \pi_2}(\gamma)\\
 &=& \Theta_{\pi_2}(\g^{-1}\cdot \gamma)\\
 &=& 0
 \end{eqnarray*}
 since $\Theta_{\pi_2}$ vanishes on $\g^{-1} \cdot T^{\varepsilon}$ by \cite[Theorem 15.1]{ADSS}.
 \end{proof}
 
 Consider the $L$-packet of $\SL_2(F)$ given by
\begin{equation}
\Pi_{\phi}(\SL_2(F)) = \{\pi_1, \pi_2, \pi_3, \pi_4\}
\end{equation}
where the representations $\pi_i$ are depth-zero supercuspidal irreducible representations of $G$ (see \eqref{eqn:packet}).
In fact, $\pi_1$ is the unique $\fw$-generic element in $\Pi_\phi(\SL_2(F))$, see  Lemma \ref{pi_1gen} and \S3.0.2.

Let  $\Theta_1$,  $\Theta_2$, $\Theta_3$, $\Theta_4$ be the distribution characters of $\pi_1$, $\pi_2$, $\pi_3$, $\pi_4$. 

Given $A \in \GL_2(\C)$ let $A_*$ denote the image of $A$ in $\PGL_2(\C)$.   Let
       \begin{equation} \label{eqn:si}
       s_1 = \left(
       \begin{array}{cc}
       1 & 0\\
       0 & -1
       \end{array}
       \right)_*
             \quad s_2 = \left(\begin{array}{cc}
       0 & 1\\-1 & 0
       \end{array}
       \right)_*
       \quad 
       s_3 = \left(
       \begin{array}{cc}
       0 & 1\\
       1 & 0
       \end{array}
       \right)_*.
       \end{equation}  
 We have
       \[
       S_\phi = \{1, s_1, s_2, s_1s_2 \}.
       \]
We will enumerate the characters of $S_\phi$ in the following way:
       \[
       \begin{tabular}{c|cccc}
       & $1$ & $s_1$ & $s_2$ & $s_1s_2$\\
       \hline
       $\rho_1$ & $1$ & $1$ & $1$ & $1$\\
       $\rho_2$ & $1$ & $1$ & $-1$ & $-1$\\
       $\rho_3$ & $1$ & $-1$ & $1$ & $-1$\\
       $\rho_4$ & $1$ & $-1$ & $-1$ & $1$
       \end{tabular}
       \]
Given 
\begin{equation} 
\iota_\phi \colon\Pi_\phi(\SL_2(F)) \to \Irr(S_\phi)
\end{equation}
 we attach to $(\phi,s)$ the  virtual character    
\begin{equation} \label{eqn:virtual}                                                       
\Theta_{\phi, s} := \sum_{\pi\in\Pi_\phi(\SL_2(F)) } \iota_\phi(\pi)(s) \cdot \Theta_{\pi}.
\end{equation}
   
One constraint is that $\iota_\phi$ must assign the trivial character 
$\rho_1$ to the $\fw$-generic representation $\pi_1$.   
We choose the labelling 
\begin{equation} 
\iota(\pi_j) : = \rho_j \quad \text{ with $j = 1,2,3,4.$}
\end{equation}
The identities below on $T^\varepsilon$ are compatible with this choice.  In that case, we have
\begin{equation} \label{eqn:endos}
\Theta_{\phi, s} =\begin{cases} \Theta_1 + \Theta_2 + \Theta_3 + \Theta_4 &\text{if $s = 1$}\\
 \Theta_1 + \Theta_2 - \Theta_3 - \Theta_4   &\text{if $ s = s_1$}\\
\Theta_1 - \Theta_2 + \Theta_3 - \Theta_4  &\text{if $s = s_2$}\\
\Theta_1 - \Theta_2 - \Theta_3 + \Theta_4    &\text{if $s = s_3$}
\end{cases}.
\end{equation}

 \begin{thm}\label{FAR-NON-REG}  Let $\gamma$ be a regular semisimple element 
in the elliptic torus $T^\varepsilon$ which is far from the identity. Then we have
  \begin{eqnarray*}
\Theta_{\phi,1\hphantom{s}}(\gamma)  &=&  -2 \psi_0(\gamma)\\
 \Theta_{\phi, s_1}(\gamma)  &=& -2\psi_0(\gamma)\\
  \Theta_{\phi, s_2}(\gamma)  &=&  0\\
 \Theta_{\phi, s_3}(\gamma)  &=& 0.
 \end{eqnarray*}    
 \end{thm} 
 \begin{proof}   This follows readily from Theorem \ref{FAR}.   
 \end{proof}

   \section{Virtual characters near the identity: the non-regular parameter}  
 
 The sign character $\sgn_\varepsilon \colon F^\times \to \{\pm 1\}$ has  been defined in \S2.0.1 and the notation $\mathrm{Im}_\varepsilon$ has  been defined in \S2.0.1.        
  We recall the definition of $f(\gamma)$ from \eqref{eqn:fgamma}.

 \begin{thm}\label{NEAR-NON-REG}  Let $\gamma$ be a regular semisimple element 
in the elliptic torus $T^\varepsilon$ which is near the identity.  Then we  have 
 \begin{eqnarray*}
 \Theta_{\phi, 1}(\gamma) &=& - 2\\
 \Theta_{\phi, s_1}(\gamma) &=& -2f(\gamma)\\
 \Theta_{\phi, s_2}(\gamma) &=& 0\\
 \Theta_{\phi, s_3}(\gamma) & = & 0
 \end{eqnarray*}
 \end{thm}
 \begin{proof}  We refer to \cite[\S4]{ADSS}, especially \cite[Lemma 4.2]{ADSS} for the description of the term $H(\Lambda', F_{\theta'})$.  
 
We now consider Theorem 15.2 in \cite{ADSS}.   We make a correction to the formula for the character values on  
$T^\varepsilon$ by deleting the $\pm$ sign inside the curly bracket: this $\pm$ sign already appeared in \cite[7th formula in Table 3]{SS}.  The corrected  formula is
 \[
 \Theta_{\pi^\pm}(\gamma) = \frac{1}{2} \left\{H(\Lambda', F_{\theta'}) \frac{\sgn_{\theta'}(\eta^{-1} \mathrm{Im}_{\theta'}(\gamma))}{d_G(\gamma)} - 1\right\}.
 \]  
 With $\gamma \in T^\varepsilon, \theta' = \varepsilon, \eta = 1$ we have
 \begin{eqnarray*}
 \Theta_1(\gamma) & = & \Theta_2(\gamma) \\
 & = & ( - f(\gamma) - 1)/2
 \end{eqnarray*}
  With $\gamma \in T^\varepsilon, \g \cdot \gamma \in T^{\varepsilon, \varpi}, \theta' = \varepsilon, \eta = \varpi$ we have
 \begin{eqnarray*}
 \Theta_4(\gamma) & = & \Theta_3(\gamma) \\
 &=&  \Theta_1(\g \cdot \gamma)\\ 
 &=& (f(\gamma) - 1)/2
 \end{eqnarray*}
 and the result follows immediately.
 \end{proof}

\emph{Note.}   The uncorrected formula in \cite{ADSS} gives $\Theta_1 + \Theta_2 = -1$ and $\Theta_3 + \Theta_4 = -1$ which implies that
$\Theta_{\phi, s_1} = 0$.   This contradicts the endoscopic character identity with $s = s_1$, see the first equation in  Theorem \ref{B}.   
 
\section{Virtual characters: regular parameters}  We recall from \eqref{eqn:reg} the definition of  the depth-zero supercuspidal representation
\[
\pi^1:= \pi(\mathbb{T}^\varepsilon, \psi).
\]
The $L$-packet is then 
\[
\Pi_\phi = \{\pi^1, \pi^2\}
\]
with $\pi^2 = \g \cdot \pi^1$.

We recall that  $\gamma$ is \emph{near the identity} if $\gamma \in T_1^\varepsilon$,
and \emph{far from the identity} if
\[
\gamma \notin Z(G)T_1^\varepsilon = T_1^\varepsilon \cup - T_1^\varepsilon.
\]   

\begin{thm}\label{vch}  Let $s$ be the non-trivial element in $S_\phi$.   Let $\gamma \in T^\varepsilon$ be a regular semisimple element which is near the 
identity or far from the identity.  Then we have
\begin{eqnarray*}
\Theta_{\phi, s} (\gamma) & = &
\begin{cases}
 -\psi(\gamma) - \psi(1/\gamma) \quad  \text{when $\gamma$ is far from the identity} \\
-2f(\gamma) \quad \text{when $\gamma$ is near the identity} 
\end{cases}
\end{eqnarray*}
\end{thm}
\begin{proof}   First case: $\gamma$ is far from the identity.    The equation 
\[
\Theta_{\pi^1}(\gamma) = -\psi(\gamma) - \psi(1/\gamma)
\]
 is a re-statement of  \cite[Theorem 14.14]{ADSS}.   We have simplified this statement  by using Lemma~\ref{farfrom} and \cite[Lemma 4.2]{ADSS}.   
 According to \cite[Theorem 14.14]{ADSS} we have
 \[
 \Theta_{\pi^2}(\gamma) = 0.
 \]
  We recall that $S_\phi = \{1,s\} = \Z /2\Z$.  Denote the trivial character of $S_\phi$ by $\rho_0$ and the non-trivial 
character by $\rho_1$.   According to Lemma \ref{pi^1gen},  $\pi^1$ is 
$\fw$-generic.   In that case, the parametrization of the $L$-packet is 
\[
\iota \colon \Pi_\phi \to \Irr(S_\phi), \quad \quad \pi^1 \mapsto \rho_0, \quad \pi^2 \mapsto \rho_1.
\]
 In that case, we have
\begin{eqnarray*}                                                       
\Theta_{\phi, s}(\gamma) &=& \sum_{\pi\in\Pi_\phi(\SL_2(F)) } \trace(\iota(\pi))(s) \cdot \Theta_{\pi}(\gamma)\\
&=& \rho_0(s) \Theta_{\pi^1}(\gamma) + \rho_1(s)\Theta_{\pi^2}(\gamma)\\
&=& \Theta_{\pi^1}(\gamma)\\
&=& -\psi(\gamma) - \psi(1/\gamma).
\end{eqnarray*}
Second case: $\gamma$ is near the identity.    In the terminology of \cite[Definition 14.1]{ADSS}, we are in the \emph{unramified case}.  According to \cite[Theorem 14.20]{ADSS} we have
 \[
 \Theta_{\pi^1}(\gamma) = c_0(\pi^1) - f(\gamma)
 \]
 where, by \cite[Definition 14.17]{ADSS}, the constant term $c_0(\pi^1)$ is 
 \[c_0(\pi^1) = -1
 \]
 and so
\[
\Theta_{\pi^1}( \gamma) = -1 - f(\gamma).
\]
Note that $\gamma \in T^\varepsilon \implies \g \cdot \gamma \in T^{\varepsilon, \varpi}$.   
With $\gamma \in T_1^{\varepsilon}$ we have
\begin{eqnarray*}
\Theta_{\pi^1}(\g \cdot \gamma) &=& c_0(\pi^1) - \frac{\sgn_{\varepsilon}(\varpi^{-1} \mathrm{Im}_\varepsilon(
\g \cdot \gamma))}{d_G(\g \cdot \gamma)}
\end{eqnarray*}
by Theorem 14.20 in \cite{ADSS}.   Using \ref{SGN}, we obtain
\[
\Theta_{\pi^1}(\g \cdot \gamma) = c_0(\pi^1) + f(\gamma)
\]
With $\gamma \in T^\varepsilon$ we also have
\begin{eqnarray*}
\Theta_{\pi^2}(\gamma)
&=& \Theta_{\pi^1}(\g \cdot \gamma)
\end{eqnarray*}
and so
\[
\Theta_{\pi^2}(\gamma) = -1 + f(\gamma).
\]
The virtual character is given by 
\begin{eqnarray*}
\Theta_{\phi, s}(\gamma) &=& \rho_0(s)\Theta_{\pi^1} (\gamma) + \rho_1(s)\Theta_{\pi^2}(\gamma) \\
&=& \Theta_{\pi^1}(\gamma) - \Theta_{\pi^2}(\gamma)\\
&=& -1 - f(\gamma) - (-1 + f(\gamma))\\
&=& -2f(\gamma) .
\end{eqnarray*}

The case when $\gamma$ is near $-1$, that is to say, $\gamma \in - T_1^\varepsilon$, is dealt with in \S9.0.1.
\end{proof}

\emph{Note on the proof}.   When stating Theorem 14.20 in \cite{ADSS}, the authors impose an additional restriction on $p$, namely their Hypothesis 1.4.   Later on the same page, they assert that Theorem 14.20 remains true without the extra hypothesis.  So we shall adhere to our standing assumption that $p>2$.

  \section{The Langlands parameter of a character of the elliptic torus $E^1$} Given a smooth character
\[
\psi\colon E^1\longrightarrow\mathbb C^\times.
\]
Let \(\sigma\) be the nontrivial automorphism of \(E/F\). Define
\[
\theta \colon E^\times \to \C^\times, \qquad \theta(x)=\psi\!\left(\frac{x}{\sigma(x)}\right).
\]
Then
\[
\theta|_{F^\times}=1,
\qquad
\theta(\sigma(x))=\theta(x)^{-1}.
\]
By local class field theory, \(\theta\) corresponds to a character
\[
\chi:\W_E\longrightarrow\mathbb C^\times.
\]

We now refer to the local version of Lemma 2.1 in the article by Brooks--Petrow \cite{BP}. This Lemma is proved in their Appendix A \cite{BP}
and provides a formula for the Langlands parameter  $\phi$ of $\psi$: interpreted on \(W_F\), the parameter of \(\psi\) is
\begin{eqnarray}\label{phieqn}
\phi_\psi(\sigma)=1\rtimes\sigma,
\qquad
\phi_\psi(h)=\chi(h)\rtimes h
\quad(h\in \W_E),
\end{eqnarray}

At this point, we specialise to the quadratic character $\psi_0$ and let $h = \tau$.   We write $\phi_0 = \phi_{\psi_0}$.    In that case, the parameter $\phi_0$ factors through the Galois group $\Gal(K/F)$, where $K/F$ is the biquadratic extension of $F$, and satisfies the two equations  
\begin{eqnarray*}
\phi_0(\sigma)&=&1\rtimes\sigma\\
\phi_0(\tau)&=&\chi_0(\tau)\rtimes\tau.
\end{eqnarray*}

\begin{lem}\label{chi0}   We have
\[
\chi_0(\tau) = -1.
\]
\end{lem}

\begin{proof}  Let \(q=|k_F|\). Since \(p>2\), the quadratic character \(\psi_0\) factors through \(k_E^1\). Thus, for \(u\in\mathfrak o_E^\times\),
\[
\theta_0(u)
=\left(\frac{\bar u}{\bar u^q}\right)^{(q+1)/2}
=\bar u^{-(q^2-1)/2}.
\]
Consequently, \(\theta_0(u)=1\) precisely when \(\bar u\) is a square in \(k_E^\times\). Also \(\theta_0(\varpi)=1\), so
\[
\ker\theta_0
=\varpi^{\mathbb Z}
 \{u\in\mathfrak o_E^\times:\bar u\text{ is a square}\}.
\]
Now \(K=E(\sqrt{\varpi})\). Every unit with square residue is a square in \(E^\times\), by Hensel's lemma, and hence a norm from \(K\). Moreover, \(-1\) is a square in \(E\), since \(|k_E|=q^2\equiv1\pmod4\). Choosing \(i\in E\) with \(i^2=-1\), we have
\[
N_{K/E}(i\sqrt{\varpi})=\varpi.
\]
Therefore \(\ker\theta_0\subseteq N_{K/E}(K^\times)\). Both groups have index two, so they are equal. Local reciprocity now identifies \(\chi_0\) with the nontrivial character of \(\operatorname{Gal}(K/E)\). Since \(\tau\) is its nontrivial element,
\[
\chi_0(\tau)=-1.
\]
\end{proof}

For \(\psi=\psi_0\), we have \(\chi_0(\tau)=-1\) by Lemma \ref{chi0}.    Thus the equations are as follows:
\begin{eqnarray}\label{phi0eqn}
\phi_0(\sigma)&=&1\rtimes\sigma\\
\phi_0(\tau)&=&(-1)\rtimes\tau \nonumber
\end{eqnarray}

\section{Endoscopy: the non-regular parameter} \subsubsection{The extended endoscopic triples}\label{endo}   We recall that $s_1, s_2$ and $s_3$ denote the images in $\PGL_2(\C)$ of the matrices
\[
\left(\begin{matrix}1&0\\0&-1\end{matrix}\right), \quad \left(\begin{matrix}0&1\\-1&0\end{matrix}\right),  \quad   \left(\begin{matrix}0&1\\1&0\end{matrix}\right).
\] 

 We recall from \S4 the definition of the non-regular parameter $\phi$.   The defining equations are
  \[
  \phi(\sigma) = s_2 \times \sigma, \quad \quad \phi(\tau) = s_1 \times \tau.
  \]
  
We will now construct  the extended endoscopic triples $\fe_1, \fe_2, \fe_3$.   Let
\[
E_1 = F(\sqrt {\varepsilon}), \qquad E_2 = F(\sqrt {\varpi}), \qquad E_3 = F(\sqrt{\varepsilon \varpi}).
\]


We consider the identity component of the centralizer of $s_k$ in $\widehat{G}$.  Following \cite[\S6.2]{T}, we will define
\[
\widehat{H}_k = \Cent(s_k, \widehat{G})^0 
\]
with $k = 1,2,3$.   Then each $\widehat{H}_k$ is a complex torus of dimension $1$.   Let 
\[
j_k : \C^\times \to \widehat{H}_k \qquad k = 1,2,3
\]
be a choice of  isomorphism sending $-1 $ to $s_k$.   We note that
\[
\quad
s_\ell j_k(z)s_\ell^{-1}=j_k(z^{-1})
\qquad(\ell\ne k).
\quad
\]
Indeed, \(s_\ell\) centralizes \(s_k\), so it normalizes
\(\widehat H_k=\operatorname{Cent}(s_k,\widehat G)^0\).
For \(\ell\ne k\), it lies outside this torus and represents its nontrivial Weyl element, which acts by inversion.

Take $k=1, \ell = 2$: then $\sigma$ acts on $\widehat{H}_1$ by inversion; dually, $\Gal(E_1/F)$ acts on $H_1$ the norm $1$ subgroup of $E_1^\times$ so $H_1$ qualifies as an endoscopic group.   

Take $k = 2, \ell = 1$: then $\tau$ acts on $\widehat{H}_2$ by inversion; dually, $\Gal(E_2/F)$ acts on $H_2$ the 
norm $1$ subgroup of $E_2^\times$ so $H_2$ qualifies as an endoscopic group.  

Take $k = 3$: then \(\sigma\) and \(\tau\) act by inversion, \(\sigma\tau\) acts trivially; this produces the splitting field $E_3$ and the norm $1$ subgroup of $E_3^\times$ qualifies as an endoscopic group. 

In summary: the groups
\[
H_k: = \{x \in E_k^\times: N_{E_k/F}(x) = 1\}
\]
with $k = 1,2,3$ are endoscopic groups.

The $L$-embeddings 
\[
{}^L\eta_k : \widehat{H}_k \rtimes \W_F \to \widehat{G} \times \W_F
\]
 are defined as follows:
\[
{}^L\eta_k(z\rtimes w)=
\begin{cases}
(j_k(z),w),&w\in W_{E_k},\\
(j_k(z)s_{k+1},w),&w\notin W_{E_k},
\end{cases}
\qquad s_4=s_1.
\]
The inversion identities  then verify the homomorphism property.    This leads to the following definition:
 \[
 \fe_k: = (s_k, H_k, {}^L\eta_k)
 \]
 with $k = 1,2,3$.   We will sometimes write $H = H_1, {}^L\eta = {}^L\eta_1, \fe = \fe_1$ so that 
 \[
 \fe = (s_1, H, {}^L\eta).
 \]

 The map $\phi^H \colon \W_F \to \C^\times \rtimes \W_F$ is defined by the equation
\[
\phi^H(\sigma) = 1 \rtimes \sigma, \quad \quad \phi^H(\tau) = (-1) \rtimes \tau.
\]

We have     
   \begin{lem}\label{H,s_1}  Let $E = F(\sqrt \varepsilon)$ and let $H$ denote the norm $1$ subgroup of $E^\times$.   Let $\psi_0$ denote the unique nontrivial quadratic character of $H$.   The non-regular parameter  $\phi$  factorizes as follows:
   \[
   \phi = {}^L\eta \circ \phi^H
   \]
   and the stable character of $\phi^H$ is given by     
     \[
     \mathscr{S}\Theta_{\phi^H} = \psi_0.
     \]
      \end{lem}
      \begin{proof}  In this context, we have the extended endoscopic triple $\fe$.     Fix an isomorphism 
      $\C^\times \cong \widehat{H}$.  Then we have       
\[
\phi^H(\sigma) = 1 \rtimes \sigma, \quad \quad \phi^H(\tau) = (-1) \rtimes \tau
\]
and
\[
{}^L\eta(1 \rtimes \sigma) = (s_2, \sigma) , \quad \quad {}^L\eta((-1) \rtimes \tau) = (s_1, \tau).
\]      
 It follows immediately that
\[
\phi = {}^L\eta \circ \phi^H.
\]

In the notation of \cite[\S4.4]{FKS} we have $(S, \theta) = (H, \psi_0)$.   According to (\ref{phi0eqn}),   $\phi^H$ is the  $L$-parameter of the quadratic character $\psi_0$.   Therefore, we  have
\[
\Theta_{\phi^H} = \psi_0
\]
and, for the stable character,
\begin{equation}\label{STheta}
\mathscr{S}\Theta_{\phi^H} = \psi_0.
\end{equation}
\end{proof}

Now we can state the next result.    
 
  \begin{lem} \label{Delta3} Let $\gamma$ be a regular semisimple element in $T^\varepsilon$.    Then we have
  \begin{eqnarray*}
   \Delta[\fw, \fe_1](\gamma, \gamma) & =&  - f(\gamma)\\
\delta \in H_2 \implies   \Delta[\fw, \fe_2](\delta, \gamma)  & =& 0\\
\delta \in H_3 \implies \Delta[\fw, \fe_3](\delta, \gamma)   &=& 0
  \end{eqnarray*}
  with $\fe_i = \fe(s_i)$ for $i = 1,2,3$.
  \end{lem} 
 \begin{proof}   Underlying the map $j_1 \colon \widehat{H} \to \widehat{G}$ we have  a canonical pair of isomorphisms of $F$-tori:
   \begin{eqnarray*}
  T^\varepsilon & \to & H\\
  \left(\begin{array}{cc} 
   a & b\\
 \varepsilon b & a
 \end{array}
 \right)
   & \mapsto &  a + b \sqrt{\varepsilon}
   \end{eqnarray*}
   and
    \begin{eqnarray*}
     \left(\begin{array}{cc} 
   a & b\\
 \varepsilon b & a
 \end{array}
 \right)
   & \mapsto &  a - b \sqrt{\varepsilon}.
   \end{eqnarray*}
   These isomorphisms are admissible in the sense of \cite[\S3]{Ka2}.     
   We therefore have two related pairs: 
  \begin{eqnarray}\label{RP}
    (\gamma, \gamma)\quad \textrm{and} \quad (1/\gamma, \gamma). 
  \end{eqnarray}    

  We will now quote the Kaletha formula for the absolute transfer factor \cite[Example 3.6.9]{Ka4}.  We will denote this transfer factor by 
 $\Delta_{\mathrm{geom}}[\mathfrak w,\mathfrak e_1]$.   Using the standard pinning and the additive character \(\Lambda\) fixed in Section~\ref{sec:depth-zero}, Kaletha's formula gives
\[
\Delta_{\mathrm{geom}}[\mathfrak w,\mathfrak e_1](c,c)
=
\epsilon\!\left(\tfrac12,\kappa_{E/F},\Lambda\right)
\kappa_{E/F}\!\left(\frac{c-\bar c}{2\eta}\right)
|c - \overline{c}|_F.
\]

With the notation  $c = a + b \sqrt \varepsilon,\: \overline{c} = a - b\sqrt \varepsilon, \: \eta = \sqrt \varepsilon$, this formula becomes 
\[
\Delta_{\mathrm{geom}}[\mathfrak w,\mathfrak e_1](c,c)
=
\epsilon\!\left(\tfrac12,\kappa_{E/F},\Lambda\right)
\kappa_{E/F}(b)
|b|_F.
\]
 
 We normalise this transfer factor as follows:
 \[
 \Delta[\fw, \fe_1] (c,c)= d_G(c)^{-2}\Delta_{\mathrm{geom}}[\fw, \fe_1](c,c)
 \]
 in order to be consistent with the transfer factor in \cite{AK}, for which there is a clean formulation of the endoscopic character identity.
    In that case, the formula for the transfer factor $\Delta[\fw, \fe_1]$ is  
 \begin{equation}\label{Deltadef}
\Delta[\mathfrak w,\mathfrak e_1](c,c)
=
\epsilon\!\left(\tfrac12,\kappa_{E/F},\Lambda\right)
\kappa_{E/F}(b)
|b|_F^{-1}
\end{equation}
since $d_G(c) = |b|_F$.     

 Since \(\kappa_{E/F}\) is unramified quadratic and \(\Lambda\) is trivial on \(\mathfrak p_F\) but nontrivial on \(\fo_F\), we have 
 \[
\epsilon\!\left(\tfrac12,\kappa_{E/F},\Lambda\right)=-1.
\]
This follows from the first proposition in \cite[\S23.5]{BH}.
We therefore have
\[
\begin{aligned}
\Delta[\mathfrak w,\mathfrak e_1](c,c)
&=-\kappa_{E/F}(b)|b|_F^{-1}\\
&=-(-1)^{v(b)}q^{v(b)}\\
&=-(-q)^{v(b)}\\
&=-f(c).
\end{aligned}
\]
We have shown that  
\begin{equation}\label{Delta7}
\Delta[\fw, \fe_1](\gamma , \gamma) = - f(\gamma).
\end{equation}

Inversion on $H_1$ is an automorphism of the extended datum $\fe_1$.  Invariance of the transfer factor under this automorphism gives 
$\Delta[\fw, \fe_1](1/\gamma,\gamma) = \Delta[\fw, \fe_1](\gamma,\gamma)$.  Combining this with (\ref{Delta7}) gives 
\begin{equation*}\label{DeltaFormula}
\Delta[\fw, \fe_1](1/\gamma, \gamma) = - f(\gamma).
\end{equation*}
At this point, write $H = H_1$ and $E = E_1$.   Note that $H_1$ (resp. $H_2, H_3$) is the endoscopic group $U(1)$ split over $E_1$ (resp. $E_2, E_3$), as in \cite[\S11.1.5]{CFM+}.

There is no admissible isomorphism from $T^\varepsilon$ to $H_2$.   If $\delta \in H_2$ and $\gamma \in T^\varepsilon$ then $\gamma $ and $\delta$ are not related and so 
\[
\Delta[\fw, \fe_2](\delta, \gamma) = 0.   
\]
Similarly, there is no admissible isomorphism from $T^\varepsilon$ to $H_3$.   If $\delta \in H_3$ and $\gamma \in T^\varepsilon$ then $\gamma$ and $\delta$ are not related and so 
\[
\Delta[\fw, \fe_3](\delta, \gamma) = 0.   
\]
\end{proof}

\begin{thm}\label{abc}  Let $\phi$ be the non-regular supercuspidal Langlands parameter for $\SL_2$.   If $\gamma$ is a regular semisimple element in  $T^\varepsilon$ then  we have  
\begin{eqnarray*}
\mathscr{E}(\phi, s_1:\gamma) &=& -2f(\gamma) \psi_0(\gamma)\\
\mathscr{E}(\phi, s_2:\gamma) &=& 0\\
\mathscr{E}(\phi, s_3:\gamma) &=& 0.
 \end{eqnarray*}
 \end{thm}
 \begin{proof}   We have
   \begin{eqnarray*}
   \mathscr{E}(\phi, s_1: \gamma) &=& \sum_{\delta \in H(F)} \Delta[\fw, \fe_1](\delta, \gamma) \, \mathscr{S}\Theta_{\phi^H} (\delta)\\
    &=& \sum_{\delta \in H(F)} \Delta[\fw, \fe_1](\delta, \gamma)
 \psi_0(\delta) \quad \textrm{by} \quad (\ref{STheta})\\ 
 &=&  \Delta[\fw, \fe_1](\gamma, \gamma)  \psi_0(\gamma) 
  + \Delta[\fw, \fe_1](1/ \gamma, \gamma) \psi_0( 1/\gamma)\\
  &=& - f(\gamma) \psi_0(\gamma) - f(\gamma) \psi_0(1/\gamma) \quad \textrm{by Lemma} \:\, \ref{Delta3}\\
  &=& -2 f(\gamma) \psi_0(\gamma) \quad \textrm{since} \quad \psi_0 = \psi_0^{-1}.
    \end{eqnarray*} 
        We also have
     \begin{eqnarray*}
   \mathscr{E}(\phi, s_2: \gamma) &=& \sum_{\delta \in H(F)} \Delta[\fw, \fe_2](\delta, \gamma) \, \mathscr{S}\Theta_{\phi^H} (\delta)\\
    &=& 0 \quad \textrm{by Lemma} \:\, \ref{Delta3} 
        \end{eqnarray*} 
 and       
          \begin{eqnarray*}
   \mathscr{E}(\phi, s_3: \gamma) &=& \sum_{\delta \in H(F)} \Delta[\fw, \fe_3](\delta, \gamma) \, \mathscr{S}\Theta_{\phi^H} (\delta)\\
    &=& 0 \quad \textrm{by Lemma} \:\, \ref{Delta3}
            \end{eqnarray*} 
  \end{proof}

 Far from the identity, the right-hand-side of Theorem \ref{abc} is  
\[
 - 2\psi_0(\gamma) 
\]
since $f(\gamma) = 1$ by Lemma \ref{farfrom}.   Near the identity, the right-hand-side of Theorem \ref{abc} is 
\[
 - 2f(\gamma)
\]
since $\gamma$ and $1/\gamma$ are both near the identity and $\psi_0$ has depth zero.   This mirrors the calculations in \cite{ADSS}: in that paper, the character formulas are  done separately, according as $\gamma$ is far from the identity or near the identity.   

\section{Endoscopy: regular parameters}  
Let \(H=H_1\), and let \(\psi\) be a depth-zero character of \(H(F)\) satisfying \(\psi\ne\psi^{-1}\). Let 
\[
\chi \colon \W_E \to \C^\times
\]
 be the character of $\W_E$ attached to $\psi$ as in \S7.  Write \(\phi^H:W_F\to{}^LH\) for the Langlands parameter of $\psi$, as in (\ref{phieqn}).  Choose a lift \(\sigma\in W_F\) of the nontrivial element of $\Gal(E/F)$.  Then we have
\begin{eqnarray*}
\phi^H(\sigma) &=& 1\rtimes \sigma\\
\phi^H(h) &=& \chi(h) \rtimes h \quad \textrm{with} \quad h \in \W_E.
\end{eqnarray*}
 
 The $L$-embedding ${}^L\eta$ was already defined in \S \ref{endo}.   Now we define
\[
\phi:={}^L\eta\circ\phi^H.
\]
 By the Langlands correspondence for tori,
\[
\varphi_H(\sigma)=1\rtimes\sigma.
\]
Since \(\sigma\notin W_E\), we have
\[
\begin{aligned}
\varphi(\sigma)
&={}^L\eta\bigl(\varphi_H(\sigma)\bigr)\\
&={}^L\eta(1\rtimes\sigma)\\
&=(s_2,\sigma).
\end{aligned}
\]
We also have, with $h \in \W_E$, 
\[
\begin{aligned}
\varphi(h)
&={}^L\eta\bigl(\varphi_H(h)\bigr)\\
&={}^L\eta(\chi(h) \rtimes h)\\
&=(j_1(\chi(h)),h)
\end{aligned}
\]

 
\begin{lem} We have
\[
\Pi_\phi(\mathrm{SL}_2(F))=\{\pi^1,\pi^2\},
\]
where \(\pi^1\) and \(\pi^2\) are the regular depth-zero supercuspidal representations defined in (3.1).
\end{lem}
\begin{proof} The nontrivial element in $\Gal(E/F)$  acts by inversion on $H$.   Since \(\psi\ne\psi^{-1}\), the pair \((H,\psi)\) is a regular depth-zero elliptic pair. By the regular depth-zero packet construction of [Ka2, \S3.4, \S5.2], the packet attached to \(\phi={}^L\eta\circ\phi_H\) consists of the representations associated to the \(G(F)\)-conjugacy classes of embeddings of this pair in its stable conjugacy class.
These embedding classes are parametrized, after choosing a reference embedding, by
\[
\ker\!\left(H^1(F,H)\longrightarrow H^1(F,G)\right).
\]
For \(G=\mathrm{SL}_2\), we have \(H^1(F,G)=1\). The norm sequence and Hilbert 90 give
\[
H^1(F,H)\cong F^\times/N_{E/F}(E^\times)\cong\mathbb Z/2.
\]
The last isomorphism follows because \(E/F\) is unramified quadratic: norms contain every unit and have even valuation.
The two embedding classes are represented by the elliptic tori at the hyperspecial vertices $x_1$, $x_2$ of Section~\ref{sec:depth-zero}. For each embedding, the depth-zero construction inflates the irreducible Deligne--Lusztig representation associated to the reduction of \(\psi\), and compactly induces from the corresponding parahoric. Consequently the resulting representations are
\[
\mathrm{c\!-\!Ind}_{G_{x_1}}^G|R_{T_\epsilon,\psi}|=\pi^1,
\qquad
\mathrm{c\!-\!Ind}_{G_{x_2}}^G|R_{T_\epsilon,\psi}|=\pi^2.
\]
These exhaust the packet, proving the assertion. 
\end{proof}

 \begin{thm}\label{ENDO}  Let $\gamma$ be a regular semisimple element in $T^\varepsilon$.      We have 
  \[
  \mathscr{E}(\phi, s:\gamma) =   - f(\gamma)( \psi(\gamma) + \psi(1/ \gamma)) .
  \]
    \end{thm}
  \begin{proof}  Let $s = s_1$ and let   
\[
\fe = (s, H, {}^L\eta)
\]
be the extended endoscopic triple from \S8.
We note that this extended endoscopic triple is \emph{identical} to the one constructed for the non-regular parameter. In the two cases (regular parameter, non-regular parameter) the endoscopic data are the same, but the little parameters $\phi^H$ are different.

Next, we determine the stable character of the little parameter $\phi^H$.   Since $\phi^H$ is the $L$-parameter of $\psi$, the $L$-packet $\Pi_{\phi^H}$ is the set containing the one element $\psi$:
\[
\Pi_{\phi^H} = \{\psi\}.
\]

Therefore, we have
\begin{eqnarray}\label{SS}
\mathscr{S}\Theta_{\phi^H} &= &\Theta_{\phi^H,1} \nonumber\\
&=& \psi.
\end{eqnarray}


   The transfer factor   is given by equation (\ref{Delta7}), namely
   \begin{equation}\label{DeltaFormula2}
   \Delta[\fw, \fe](\gamma,\gamma ) = \Delta[\fw, \fe](1/\gamma, \gamma) = - f(\gamma).
   \end{equation}
  Finally, we have  
   \begin{eqnarray*}
    \mathscr{E}(\phi, s:\gamma) & = &  \sum_{\delta \in H(F)} \Delta[\fw, \fe](\delta, \gamma) \, \mathscr{S}\Theta_{\phi^H} (\delta)\\
     &=& \sum_{\delta \in H(F)} \Delta[\fw, \fe](\delta, \gamma)
 \psi(\delta) \quad \textrm{by} \quad (\ref{SS})\\
 &=& \Delta[\fw, \fe](\gamma, \gamma)\psi(\gamma) 
  + \Delta[\fw, \fe](1/ \gamma, \gamma) \psi(1/ \gamma)
 \quad \textrm{by} \quad (\ref{RP})\\
  &=& - f(\gamma)(\psi(\gamma)  + \psi(1/ \gamma)) \quad \textrm{by} \quad (\ref{DeltaFormula2})
  \end{eqnarray*}  
   as required. 
  \end{proof}
  
\subsubsection{Completion of the proofs of Theorems \ref{A} and \ref{B}}.  The distinction in this article between \emph{near the identity} and \emph{far from the identity} leaves out elements near $-1$.   In order to remedy this omission, we proceed as follows.    

Let \(\gamma\in -T_1^\varepsilon\) be regular semisimple, and put 
\[
\eta=-\gamma\in T_1^\varepsilon. 
\]
 Then \(\eta\) is also regular semisimple. Both belong to the domain of \(f\), and
\[
f(-\gamma)=(-q)^{v(-b)}=(-q)^{v(b)}=f(\gamma).
\]

 Let $\psi$ be the character of $T^\varepsilon$ which features in Theorem \ref{A}.   

By the transfer-factor formula and the two matching elements,
\[
\begin{aligned}
\mathcal E(\phi,s_1:\gamma)
&=-f(\gamma)\bigl(\psi(\gamma)+\psi(\gamma^{-1})\bigr)\\
&=-\psi(-1)f(\eta)
      \bigl(\psi(\eta)+\psi(\eta^{-1})\bigr)\\
&=\psi(-1)\mathcal E(\phi,s_1:\eta).
\end{aligned}
\]
Each representation in the packet has central character taking the value \(\psi(-1)\) at \(-1\). Consequently,
\[
\Theta_{\phi,s_1}(\gamma)
=\psi(-1)\Theta_{\phi,s_1}(\eta).
\]
The character identity already established at \(\eta\) therefore implies the identity at \(\gamma\).   This completes the proof of Theorem \ref{A}.   

In order to complete the proof of Theorem 1.8, replace \(\psi\) by \(\psi_0\); the vanishing identities for \(s_2,s_3\) extend by the same central-character argument.

\markright{\textsc{References}}
\markleft{\textsc{References}}

\end{document}